\documentclass[reqno]{amsart}
\usepackage{amsfonts,amsmath,amssymb,color}

\numberwithin{equation}{section}

\usepackage[babel=true,kerning=true]{microtype}

\usepackage[top = 4cm, left = 4cm, right = 4cm]{geometry}

\usepackage{amsthm,leqno}
\usepackage{hyperref}
\usepackage{bigints}

\usepackage{amsfonts,amsmath,amssymb,enumerate}
\usepackage{graphicx}
\usepackage{tikz}
\usetikzlibrary{decorations.markings}
\usetikzlibrary{plotmarks}
\usetikzlibrary{patterns}

\newtheorem{theo}{Theorem}[section]
\newtheorem{prop}[theo]{Proposition}

\newtheorem{lemme}[theo]{Lemma}
\newtheorem{corol}[theo]{Corollary}
\newtheorem{claim}{Claim}

\theoremstyle{remark}

\newcommand{\be}{\begin{equation*}}
\newcommand{\ee}{\end{equation*}}
\newcommand{\ben}{\begin{equation}}
\newcommand{\een}{\end{equation}}
\newcommand{\begincal}{\begin{eqnarray*}}
\newcommand{\fincal}{\end{eqnarray*}}
\newcommand{\bal}{\begin{aligned}}
\newcommand{\eal}{\end{aligned}}

\newcommand{\ua}{u_\alpha}

\newcommand{\Lax}{\Lambda_{\xi}}

\newcommand{\mk}{\mu_k}
\newcommand{\xk}{\xi_k}
\newcommand{\exk}{\textrm{exp}_{\xi_k}^{g_{\xi_k}}}
\newcommand{\vek}{\varepsilon_k}

\newcommand{\RR}{\mathbb{R}}

\newcommand{\ve}{\varepsilon}
\newcommand{\vp}{\varphi}

\numberwithin{equation}{section}

\begin{document}

\title[Non-compactness of the Einstein-Lichnerowicz equation]{Non-compactness and Infinite number of conformal initial data sets in high dimensions} 

\author{Bruno Premoselli \and Juncheng Wei}

\thanks{J. Wei was supported by NSERC of Canada.} 
 
\address{Bruno Premoselli, Laboratoire AGM, Universit\'e de Cergy-Pontoise site Saint-Martin,
2 avenue Adolphe Chauvin 
95302 Cergy-Pontoise cedex 
France. }
\email{bruno.premoselli@u-cergy.fr \\ }

\address{Juncheng Wei, Department of Mathematics, University of British Columbia
                   Vancouver V6T 1Z2, Canada.}
\email{jcwei@math.ubc.ca}

\begin{abstract} 
On any closed Riemannian manifold of dimension greater than $7$, we construct examples of background physical coefficients for which the Einstein-Lichnerowicz equation possesses a non-compact set of positive solutions. This yields in particular the existence of an infinite number of positive solutions in such cases. 
\end{abstract}

\maketitle  
\section{Introduction}

\noindent Let $(M,g)$ be a closed compact Riemannian manifold of dimension $n \ge 6$. We investigate in this work non-compactness issues for the set of positive solutions of the Einstein-Lichnerowicz equation in $M$, which writes as follows:
\ben \label{intro1}
 \triangle_g u + h u = f u^{2^*-1} + \frac{a}{u^{2^*+1}}.
 \een 
Here $h,f,a$ are given functions in $M$ such that $\triangle_g + h$ is coercive, $f > 0$ and $a \ge 0$, $a \not \equiv 0$ and $2^* = \frac{2n}{n-2}$ is the critical exponent for the embedding of the Sobolev space $H^1(M)$ into Lebesgue spaces. Equation \eqref{intro1} arises in the initial-value problem in General Relativity, when one looks for initial-data sets for the Einstein equations \emph{via} the conformal method. The determination of constant-mean-curvature initial data sets amounts to the resolution of equation \eqref{intro1} in the prominent case where the coefficients $h,f,a$ take the following form:
\begin{equation} \label{coefficients}
\begin{aligned}
& h  = c_n \left( S_g - |\nabla \psi|_g^2 \right) ~~,~~f  =  c_n \left( 2 V(\psi) - \frac{n-1}{n} \tau^2 \right)  ~~,~~ a = \pi^2,  \\
\end{aligned} \end{equation}
where $c_n = \frac{n-2}{4(n-1)}$. In \eqref{coefficients}, $V : \RR \to \RR$ is a potential, $\psi: M \to \RR$ is a scalar-field, $\tau \in \RR$ is the mean curvature and $S_g$ is the scalar curvature of $g$. Physically speaking, $\psi$ and $\pi$ represent respectively the restriction of the ambient scalar-field and of its time-derivative to $M$ and $\tau$ is the mean curvature of the Cauchy hypersurface $M$ embedded in the space-time. See Bartnik-Isenberg \cite{BarIse} for a survey reference on the constraint equations.

\medskip
\noindent Following the terminology introduced in Premoselli \cite{Premoselli4}, in this work we consider equation \eqref{intro1} in the so-called \emph{focusing case}, defined as:
\ben \label{deffocusing}
\textrm{ focusing case: } f > 0 \textrm{ in } M.
\een
Since $a \ge 0$, standard variational arguments show that the coercivity of $\triangle_g + h$ is a necessary condition for \eqref{intro1} to possess smooth positive solutions. The existence of solutions of \eqref{intro1} in the focusing case was first obtained in Hebey-Pacard-Pollack \cite{HePaPo} and multiplicity issues were later on investigated in Ma-Wei \cite{MaWei}, Premoselli \cite{Premoselli2} and Holst-Meier \cite{HolstMeier}. Existence results for the non-constant-mean-curvature generalization of \eqref{intro1} in the physical case \eqref{coefficients} are in Premoselli \cite{Premoselli1}. 

\medskip

\noindent Stability issues for equation \eqref{intro1} in the focusing case have been investigated in Druet-Hebey \cite{DruHeb}, Hebey-Veronelli \cite{HebeyVeronelli} and Premoselli \cite{Premoselli2}. The stability of the general conformal constraint system in all dimensions has been investigated by Druet-Premoselli \cite{DruetPremoselli} and Premoselli \cite{Premoselli4}. In the specific case of equation \eqref{intro1}, the results of Druet-Hebey \cite{DruHeb} and Premoselli \cite{Premoselli4} yield in particular that in dimensions $n \ge 6$, equation \eqref{intro1} is stable with respect to perturbations of its coefficients as soon as there holds in $M$:
\ben \label{condstabilite}
h - c_n S_g + \frac{(n-2)(n-4)}{8(n-1)} \frac{\triangle_g f}{f} < 0,
\een
where $S_g$ denotes the scalar curvature of $g$. More precisely, if \eqref{condstabilite} holds, then for any sequences $(h_\alpha)_\alpha, (f_\alpha)_\alpha$ and $(a_\alpha)_\alpha$, with $a_\alpha \ge 0$ satisfying:
\[ \Vert h_\alpha - h \Vert_{C^0(M)} + \Vert f_\alpha - f\Vert_{C^2(M)} + \Vert a_\alpha - a\Vert_{C^0(M)} \to 0\]
as $\alpha \to + \infty$, and for any sequence $(\ua)_\alpha$ of positive solutions of:
\ben \label{eqperturbe}
\triangle_g \ua + h_\alpha \ua = f_\alpha \ua^{2^*-1} + \frac{a_\alpha}{\ua^{2^*+1}},
\een
there holds, up to a subsequence, that $\ua$ converges to some positive solution $u_0$ of \eqref{intro1} in $C^{1,\eta}(M)$ for all $0 < \eta < 1$. In dimensions $3 \le n \le 5$, Druet-Hebey \cite{DruHeb} showed that stability always holds if $a \not \equiv 0$. For general references on the notion of stability for critical elliptic equations, see Druet \cite{DruetENSAIOS} and Hebey \cite{HebeyZLAM}. The notion of elliptic stability for equation \eqref{intro1} in the above sense yields structural informations on the set of positive solutions of \eqref{intro1}, and as such is fundamental in order to understand its properties. The stability of \eqref{intro1} was for instance crucially used in Premoselli \cite{Premoselli2} to describe multiplicity issues for \eqref{intro1} in small dimensions. In the physical case of the conformal method, where the coefficients are given by \eqref{coefficients}, the stability and the instability of \eqref{intro1} reformulate in terms of the relevance of the conformal method in the determination of initial data sets. These issues are discussed in detail in Premoselli \cite{Premoselli4}.

\medskip

\noindent In this work we establish the sharpness of assumption \eqref{condstabilite} for the specific physical Einstein-Lichnerowicz equation. In dimensions $n \ge 6$, Druet-Hebey \cite{DruHeb} have proven that equation \eqref{intro1} is not stable in general by showing the existence of sequences of coefficients $(h_\alpha)_\alpha, (f_\alpha)_\alpha$, $(a_\alpha)_\alpha$ for which \eqref{eqperturbe} possesses blowing-up sequences of solutions on specific manifolds. These sequences, however, did not have the physical form \eqref{coefficients}. We extend here their results to the general inhomogeneous setting of an arbitrary Riemannian manifold of dimension $n \ge 6$ and exhibit non-compactness phenomena for \eqref{intro1}. In particular, when $n \ge 7$, we obtain non-compactness results for the Einstein-Lichnerowicz equation of the conformal method, where $h,f$ and $a$ are given by the restrictive expression \eqref{coefficients}. Our main result is as follows:

\begin{theo} \label{Th1}
Let $(M,g)$ be a closed $n$-dimensional Riemannian manifold, with $n \ge 7$. There exist examples of a $C^2$ potential $V$, a mean curvature $\tau \in \RR$, a $C^2$ scalar-field $\psi$ and a $C^2$ function $\pi$ which are focusing in the sense of \eqref{deffocusing} (with $f$ given by \eqref{coefficients}), such that the Einstein-Lichnerowicz equation:
\ben \label{intro2}
 \triangle_g u + c_n \left( S_g - |\nabla \psi|_g^2 \right) u = c_n \left( 2 V(\psi) - \frac{n-1}{n} \tau^2 \right)u^{2^*-1} + \frac{\pi^2}{u^{2^*+1}} 
 \een
admits a sequence of solutions $(u_k)_k$ that satisfy $\Vert u_k \Vert_{L^\infty(M)} \to + \infty$ as $k \to +\infty$. In addition, the $u_k$ are all different, possess a single blow-up point and blow-up with a non-zero limit profile. 
\end{theo}
\noindent Theorem \ref{Th1} shows in particular that the set of positive, physical solutions of \eqref{intro2} is not compact in the $C^2(M)$-topology and that equation \eqref{intro2} has therefore an infinite number of positive solutions. We state it as a Corollary:

\begin{corol} \label{corolmulti}
Let $(M,g)$ be a closed $n$-dimensional Riemannian manifold, with $n \ge 7$. There exist examples of scalar-field data $\psi$ and $\pi$, of a potential $V$ and of a mean curvature $\tau$ such that the Einstein-Lichnerowicz equation \eqref{intro2} possesses an infinite number of solutions.
\end{corol}
\noindent In Corollary \ref{corolmulti}, in view of Theorem \ref{Th1}, the solutions are obtained as a sequence that blows-up at a critical point of the scalar-field $\psi$. In particular, in view of the physical origin of equation \eqref{intro2} in the analysis of the constraint system, Corollary \ref{corolmulti} shows that when condition \eqref{condstabilite} is not satisfied, the conformal constraint system of equations possesses an infinite number of solutions. This yields in turn an infinite number of initial data sets. Here again we refer to the discussion in Premoselli \cite{Premoselli4} for the geometric and physical implications of this fact.

\medskip

\noindent In dimension $6$, we have the following counterpart:
\begin{theo} \label{Th2}
Let $(M,g)$ be a $6$-dimensional Riemannian manifold. There exist examples of $C^2$ function $h,f$ and $a$ such that the equation
\[ \triangle_g u+ h u = f u^{2^*-1} + \frac{a}{u^{2^*+1}}\]
admits a sequence of solutions $(u_k)_k$ that blow-up in the $L^\infty$-norm. In addition, the $u_k$ are all different, possess a single blow-up point and blow-up with a non-zero limit profile.
\end{theo}
\noindent Again, Theorem \ref{Th2} yields a non-compactness result for the set of positive solutions of \eqref{intro1} in the $C^2(M)$-topology and shows that, for a suitable choice of $h,f$ and $a$, equation \eqref{intro1} possesses an infinite number of solutions. 

\medskip
\noindent Theorems \ref{Th1} and \ref{Th2} show that condition \eqref{condstabilite} is sharp in order for \eqref{intro1} to be stable. They deal with the high-dimensional case of equation \eqref{intro1}. As already stated, note also that, as proven in Druet-Hebey \cite{DruHeb},
equation \eqref{intro1} is always stable in dimensions $3$ to $5$ when $a \not \equiv 0$.

\medskip

\noindent We prove Theorem \ref{Th1} by performing a finite-dimensional Lyapunov-Schmidt reduction in Sections \ref{setting} to 
 \ref{conclusion} below. The main problem when dealing with the physical equation where the coefficients are given by \eqref{coefficients} is that the coefficient $h$ always lies below the geometric threshold of the scalar curvature. We therefore take into account a non-constant function $f$ to perform the finite-dimensional reduction and this forces us to obtain $C^1$-uniform asymptotic expansions in the computation of the energy to conclude. In Section \ref{setting} we set the problem and give the explicit expression of the functions $V$, $\psi$, $\pi$ and $\tau$ chosen. In section \ref{findim} we describe the blow-up profiles we will work with and perform a Lyapunov-Schmidt finite-dimensional reduction. Sections \ref{energy} and \ref{reste} are devoted to the proof of $C^1$-uniform asymptotic expansions of the reduced energy and of the remainder term. Finally, Section \ref{conclusion} gives the final argument to conclude the proof of Theorem \ref{Th1}. The proof of Theorem \ref{Th2} is given in Section \ref{negale6}.

\section{Setting of the problem} \label{setting}

The proof of Theorem \ref{Th1} relies on a finite-dimensional reduction method. We describe in this section its setting. Let $n \ge 7$ and $(M,g)$ be a $n$-dimensional Riemannian manifold. Let $\xi_0 \in M$ be some fixed point in $M$. Assume that $|W(\xi_0)|_g > 0$ if $(M,g)$ is not locally conformally flat, where $W(\xi_0)$ denotes the Weyl tensor of $g$ at $\xi_0$. On some neighborhood $U$ of $\xi_0$ we can find a smooth determination $\xi \mapsto (e_1(\xi), \cdots, e_n(\xi))$ of an orthonormal basis of $T_\xi M$. In the following, for $\xi \in U$, the notation $\textrm{exp}_\xi^{g_\xi}$ will denote the exponential map for the metric $g_\xi$ at point $\xi$ with the identification of $T_\xi M$ to $\RR^n$ via the field $(e_1, \cdots, e_n)$. The standard conformal normal coordinates theorem of Lee-Parker \cite{LeeParker} asserts that there exists $\Lambda \in C^\infty(M\times M)$ such that for any point $\xi \in M$ there holds, for some arbitrarily large integer $N$:
\ben \label{confnorm}
\left| \left( \textrm{exp}_\xi^{g_\xi} \right)^* g_{\xi} \right|(y) = 1 + O(|y|^N),
\een
$C^1$-uniformly in $\xi \in M$ and in $y \in T_\xi M$ in a small geodesic ball for the metric $g_\xi$. In \eqref{confnorm} we have let 
\ben \label{metconforme}
g_\xi = \Lambda_\xi^{\frac{4}{n-2}}g,
\een 
where the conformal factor  $\Lambda_\xi = \Lambda(\xi, \cdot)$ can in addition be chosen to satisfy:
\ben \label{propLambda}
\Lax(\xi) = 1, \quad \nabla \Lax (\xi) = 0 .
\een

\medskip

\noindent We assume in the following that $(M,g)$ is of positive Yamabe type, that is, such that $\triangle_g + c_n S_g$ is coercive, where $c_n$ is as in \eqref{coefficients}. Let $\vp $ be a smooth positive function in $M$ and let $\tilde g = \vp^{\frac{4}{n-2}} g$. By the conformal covariance property of the conformal laplacian there holds, for any smooth positive function $u_0$: 
\be
\left( \triangle_g + c_n S_g \right) (\vp u_0) = \vp^{2^*-1} \left( \triangle_{\tilde g} + c_s S_{\tilde g}\right) u_0.
\ee
Therefore, it is easily seen that for any smooth positive function $u_0$, the function $\vp u_0$ solves \eqref{intro2} for some given background coefficients $(V, \psi, \tau, \vp^{2^*}\pi)$ if and only $u_0$ solves:
\[ \triangle_{\tilde g} u_0 + c_n \left( S_{\tilde g} - | \nabla \psi |_{\tilde g}^2\right) u_0 = c_n \left( 2 V(\psi) - \frac{n-1}{n} \tau^2 \right)u_0^{2^*-1} + \frac{ \pi^2}{u_0^{2^*+1}}. \] 
In the following we may therefore assume that $S_g$ is a positive constant satisfying:
\ben \label{condSg}
c_n S_g > 2.
\een
Let $\pi_0^2 \equiv  c_n S_g - 1 > 1$ and let $u_0 \equiv 1$. Then $u_0$ solves, in $M$:
\ben \label{ELbase}
\triangle_g u_0 + c_n S_g u_0 = u_0^{2^*-1} + \frac{\pi_0^2}{u_0^{2^*+1}}.
\een
Because of \eqref{condSg} it is easily seen that \emph{$u_0$ is strictly stable}, i.e. that there exists a positive constant $C$ such that, for any $\vp \in H^1(M)$:
\ben \label{u0nondeg}
\int_{M} |\nabla \vp|_g^2 + \left[ c_n S_g  - (2^*-1) u_0^{2^*-2} + (2^*+1) \frac{\pi_0^2}{u_0^{2^*+2}} \right] \vp^2 dv_g \ge C \Vert \vp \Vert_{H^1}. 
\een

\medskip

\noindent Let $\beta \in C^{\infty}(\mathbb{R})$ be a compactly supported function, with support contained in  $ [-M-1,M+1]$ and satisfying $\beta \equiv 1$ on $[-M, M]$, for some positive $M$ to be chosen later. Let $\Psi$ in $\RR^n$ be given by:
\ben \label{defPsi}
\Psi(x) = - |x|^2 \beta(|x|^2). 
\een
Let $(\vek)_k$ be a sequence of positive real numbers which converge towards zero as $k \to \infty$. We define a sequence $(\mk)_k$ as follows:
\ben \label{defmk}
\mk = \left \{
\bal
& \vek^{\frac{2}{n-2}} &\textrm{ if } (M,g) \textrm{ is l.c.f. or if } 7 \le n \le 9 \\
& \vek^{\frac{1}{4}}  &\textrm{ if } n \ge 10 \textrm{ and } (M,g) \textrm{ is not l.c.f.}.\\
\eal
\right.
\een
Let $(r_k)_k$ be a sequence of positive real numbers which converges towards zero as $k \to \infty$. We assume that the following relations hold:
\ben \label{condrk}
\bal
r_k  = o(k^{-2}) \quad \textrm{ and } \quad \mk = o (r_k^3) \\
\eal
\een
as $k \to + \infty$. Examples of sequences $(\vek)_k$ and $ (r_k)_k$ that satisfy \eqref{condrk} are for instance:
\[ \vek = k^{-4(n-2)} \textrm{ and } r_k = k^{-\frac{7}{3}}. \]
We define a sequence $(\xk)_k$ of points of $M$ concentrating at $\xi_0$ as above by:
\ben \label{defxik}
\xi_k = \exp_{\xi_0}^{g_{\xi_0}} \left( \left( \frac{1}{k}, 0, \cdots 0 \right)\right).
\een
Then $\xi_k \to \xi_0$ as $k \to + \infty$ and there holds: $d_g(\xk, \xi_{k+1}) \sim \frac{1}{k^2}$ as $k \to +\infty$.
Inspired by the non-compactness constructions in Brendle \cite{Bre} and Brendle-Marques \cite{BreMa} we let in what follows, for any $x \in M$:
\ben \label{defPsi0}
\Psi_0(x) = \sum_{k \ge k_0} \vek \Psi \left( \frac{1}{\mu_k} \left(\exk \right)^{-1} (x) \right),
\een
where $\Psi$ is as in \eqref{defPsi}, $\mk$ is as in \eqref{defmk} and where $k_0 >0$ is some fixed integer, to be chosen large enough. We define, for any $x \in M$, the following functions $h,f$ and $\pi$ by:
\ben \label{coeffsEL}
\left \{ \bal
h & = c_n \left( S_g  - |\nabla \Psi_0(x)|_g^2 \right), \\
f & = 1 + \Psi_0, \\
\pi & = \left(\pi_0^2 - c_n |\nabla \Psi_0|_g^2 - \Psi_0   \right)^{\frac{1}{2}}, \\
\eal \right.
\een
where $\Psi_0$ is as in \eqref{defPsi0}. Note that by \eqref{condrk}, up to assuming $k_0$ large enough, $\pi$ and $f$ in \eqref{coeffsEL} are positive functions in $C^2(M)$. It is easily seen with \eqref{ELbase} that $u_0$ is a solution of the following equation in $M$:
\ben \label{EL}  \tag{$EL$}
\triangle_g u + hu = f u^{2^*-1} + \frac{\pi^2}{u^{2^*+1}},
\een
and that, because of \eqref{u0nondeg}, $u_0$ as a solution of \eqref{EL} is still strictly stable up to choosing $k_0$ large enough. Note that the functions $h,f$ and $\pi$ defined in \eqref{coeffsEL} take the physical form given by \eqref{coefficients} for the choice of a scalar-field given by $\Psi_0$ and for suitably chosen constant $\tau$ and linear potential $V(\Psi)$.  

\medskip

\noindent 
We aim at constructing sequences of solutions of \eqref{EL} developing one bubble around each $\xi_k$ as in \eqref{defxik}, for $k$ large enough. As a first task, we take care of the negative nonlinearity in \eqref{EL}. For any $\ve >0$, we define $\eta_\ve$ in $\RR$ as follows:
\ben \label{defetaeps}
\eta_\ve(r) = \left \{
\bal
& \ve \textrm{ if } r < \ve \\
& r \textrm{ if } r \ge \ve, \\ 
\eal \right.
\een
and introduce the following truncation of \eqref{EL}:
\ben \label{ELeps} \tag{$EL_\ve$}
\triangle_g u + hu = f u^{2^*-1} + \frac{\pi^2}{\eta_\ve(u)^{2^*+1}}.
\een
A first easy remark is the following:
\begin{lemme} \label{minor}
There exists $\ve_0 > 0$ such that for any $0 \le \ve \le \ve_0$, any $C^2$ positive solution $u$ of \eqref{ELeps} satisfies:
\[ \min_M u \ge \ve_0. \]  
In particular, for $0 \le \ve \le \ve_0$, any $C^2$ positive solution of \eqref{ELeps} is also a solution of \eqref{EL}.
\end{lemme}
\begin{proof}
Let $\ve >0$ be fixed and $G$ denote the Green's function of the operator $\triangle_g +h$ in $M$. By standard arguments (see Robert \cite{RobDirichlet}), $G$ is bounded from below by some positive constant. If $u$ is a positive solution of \eqref{ELeps} a representation formula therefore yields:
\[ \min_M u \ge \frac{1}{C} \int_M \left( f u^{2^*-1} + \frac{\pi^2}{\eta_\ve(u)^{2^*+1}} \right) dv_g \]
for some positive constant $C$. By examining the values taken by the right-hand side it is easily seen that there holds in $M$:
\[ f u^{2^*-1} + \frac{\pi^2}{\eta_\ve(u)^{2^*+1}} \ge  \frac{1}{C} \min \left( \pi^2 \ve^{-2^*-1}, f^{\frac{2^*+1}{2 \cdot 2^*}} \pi^{\frac{2^*-1}{2^*}} \right),\]
which concludes the proof up to choosing $\ve$ small enough.
\end{proof}

From now on, we let $0 < \ve \le \ve_0$ be small enough so that $u_0 \equiv 1$ solves both  \eqref{EL} and \eqref{ELeps}. We let $\eta = \eta_\ve$ be fixed. By Lemma \ref{minor} the construction of blowing-up sequences of solutions of \eqref{EL} reduces to the construction of such sequences for \eqref{ELeps}.

\section{The finite-dimensional reduction method} \label{findim}

 We introduce the energy functional associated to \eqref{ELeps}: for any $u \in H^1(M)$,
\ben \label{defJ}
J(u) = \frac{1}{2} \int_M \left( |\nabla u|_g^2 + h u^2 \right) dv_g - \frac{1}{2^*} \int_M f (u^+)^{2^*} dv_g + \frac{1}{2^*} \int_M \frac{\pi^2}{\eta(u)^{2^*}} dv_g,
\een
where, $h,f,\pi$ are given by \eqref{coeffsEL}. We endow $H^1(M)$ with the following scalar product: for any $u,v \in H^1(M)$,
\ben \label{psH1}
\langle u, v \rangle_h = \int_M \left(  \langle \nabla u,  \nabla v \rangle_g + h uv \right) dv_g.
\een
For any $J \in H^1(M)'$ we will denote by $(\triangle_g +h)^{-1}(J)$ the unique element  of $H^1(M)$ such that for any $v \in H^1(M)$ there holds:
\[ J(v) = \langle (\triangle_g +h)^{-1}(J), v \rangle_h. \]
For $t >0$ and $p \in \overline{ B_0(1)}$ we define two sequences $(\delta_k(t))_k$ and $(y_k(p))_k$ by:
\ben \label{defdkyk}
\left \{
\bal
\delta_k(t) & = \mk t \\
y_k(p) & = \exk (\mk p) ,\\
\eal
\right.
\een
where $\mk$ is as in \eqref{defmk} and $\xi_k$ is as in \eqref{defxik}. For the sake of clarity, in the computations below, the dependence in $t$ and $p$ in $\delta_k$ and $y_k$ may be omitted in the expressions since no ambiguity will occur. We let $r_0 >0$ be such that $r_0 < i_{g_\xi}(M)$ for all $\xi \in M$,  where $i_{g_\xi}$ denotes the injectivity radius of the metric $g_\xi$ given by \eqref{metconforme}. Up to choosing $k_0$ large enough in \eqref{defPsi0}, by \eqref{condrk} there holds $2r_k < r_0$ for any $k$.  We let $\chi \in C^\infty(\RR)$ be such that $\chi \equiv 1$ in $B_0(1)$ and $\chi \equiv 0$ outside of $B_0(2)$. The blow-up profiles we investigate in this work are given by the following expression: for $t >0$ and $p \in \overline{ B_0(1)}$, and for any $x \in M$:
\ben \label{bulle}
W_{k,t,p}(x) = \Lambda_{y_k}(x) \chi \left( \frac{d_{g_{y_k}}(y_k,x)}{r_k}\right) \delta_k^{\frac{n-2}{2}} \left( \delta_k^2 + \frac{f(y_k)}{n(n-2)}  d_{g_{y_k}}(y_k,x)^2 \right)^{1 - \frac{n}{2}},
\een
where $\Lambda_{y_k}$ is as in \eqref{propLambda} and $\delta_k$ and $y_k$ are given by \eqref{defdkyk}. These profiles are localized in $B_{y_k}(2r_k)$, which denotes here the geodesic ball of radius $2 r_k$ with respect to the metric $g_{y_k}$. 
For a given $k$ we let $V_{0,k}, \cdots, V_{n,k}: \RR^n \to \RR$ be given by:
\ben \label{defVki}
\bal
V_{0,k}(x) & =  \left( \frac{f(y_k)}{n(n-2)}|x|^2 - 1 \right) \left( 1 + \frac{f(y_k)}{n(n-2)}|x|^2\right)^{-\frac{n}{2}} \\
V_{i,k}(y) & = f(y_k) x_i \left( 1 + \frac{f(y_k)}{n(n-2)}|x|^2 \right)^{-\frac{n}{2}} ,\\
\eal
\een
and also define, for any $x \in M$, any $1 \le i \le n $ and any $k$:
\ben \label{defZk}
\bal
Z_{0,k,t,p}(x) & = \Lambda_{y_k}(x)  \chi \left( \frac{d_{g_{y_k}}(y_k,x)}{r_k}\right) \delta_k^{\frac{n-2}{2}}  \left(  \delta_k^2 + \frac{f(y_k)}{n(n-2)}  d_{g_{y_k}}(y_k,x)^2 \right)^{- \frac{n}{2}} \\
& \times \left( \frac{f(y_k)}{n(n-2)}d_{g_{y_k}}(y_k,x)^2 - \delta_k^2 \right) \\
Z_{i,k,t,p}(x) &=  \Lambda_{y_k}(x)  \chi \left( \frac{d_{g_{y_k}}(y_k,x)}{r_k}\right)  \delta_k^{\frac{n}{2}} \left(  \delta_k^2 + \frac{f(y_k)}{n(n-2)}  d_{g_{y_k}}(y_k,x)^2 \right)^{- \frac{n}{2}} \\
& \times f(y_k)\left \langle \left( \textrm{exp}_{y_k}^{g_{y_k}}\right)^{-1}(x), e_i(y_k) \right\rangle_{g_{y_k}(y_k)}.
\eal
\een
In \eqref{defZk}, the $(e_i)_i$ denote the field of orthonormal basis introduced in the beginning of Section \ref{setting}. Finally, we let 
\ben \label{noyau}
K_{k,t,p} = \textrm{Span} \left \{ Z_{i,k,t,p}, i=0 \dots n \right \}.
\een
Since $(V_0, \cdots, V_n)$ forms an orthonormal family for the scalar product $(u,v) = \int_{\RR^n} \langle \nabla u, \nabla v \rangle dx$ in $\RR^n$, $K_{k,t,p}$ is $(n+1)$-dimensional for $k$ large enough and the $Z_{i,k,t,p}$ are ``almost'' orthogonal. We denote by $K_{k,t,p}^{\perp}$ its orthogonal in $H^1(M)$ for the scalar product given by \eqref{psH1}.

\medskip

 We construct solutions of \eqref{ELeps} of the form $u_0 + W_{k,t,p}$.  It is easily seen that thanks to the truncation function $\eta$ defined in \eqref{defetaeps} the mapping 
\[ G: u \in H^1(M) \mapsto \int_M \frac{a}{\eta(u)^{2^*-1}} dv_g\]
 is of subcritical type, in the following sense: for any sequences $(u_l)_l, (v_l)_l$ and $(w_l)_l$ weakly converging in $H^1(M)$ towards $u,v$ and $w$, there holds: 
 \[ D^2G(u_l)(v_l,w_l) \to D^2G(u)(v,w) \]
 as $l \to + \infty$. Therefore, since $u_0$ is a smooth positive strictly stable solution of \eqref{ELeps}, the general finite-dimensional reduction theorem stated in Robert-V\'etois \cite{RobertVetois} applies and yields the following:
 \begin{prop} \label{generalreduction}
 Let $k \in \mathbb{N}$. There exists $M_k > 0$, a $C^1$ mapping $\phi_k : (0,M_k) \times \overline{B_0(1)} \to  H^1(M) $ and $C^1$ functions $\lambda_{i,k}: (0,M_k) \times \overline{B_0(1)} \to \RR$, $0 \le i \le n$, such that the function given by
 \ben \label{defuk}
  u_{k,t,p} = u_0 + W_{k,t,p} + \phi_k(t,p) 
  \een
   satisfies, for any $(t, p) \in (0, M_k) \times \overline{B_0(1)}$:
 \ben \label{soloutnoyau}
u_{k,t,p} - \left( \triangle_g + h \right)^{-1} \left( f u_{k,t,p}^{2^*-1} + \frac{\pi^2}{u_{k,t,p}^{2^*+1}}\right)  = \lambda_{0,k}(t,p) Z_{0,k,t,p} + \sum_{i = 1}^n \lambda_{i,k}(t,p) Z_{i,k,t,p},
 \een
 where $W_{k,t,p} $ is as in \eqref{bulle}  and the $Z_{i,k,t,p}$ are as in \eqref{defZk}. In addition there holds that $\lambda_{i,k}(t,p) = 0$ for all $0 \le i \le n$ -- and hence
$u_{k,t,p}$ is a solution of \eqref{ELeps} -- if and only if $(t, p)$ is a critical point of the mapping $(s,q) \mapsto J(u_{k,s,q})$, where $J$ is as in \eqref{defJ}. 
 \end{prop}
\noindent Note that the mapping $\phi_k$ given in Proposition \ref{generalreduction} satisfies in addition, for all $(t,p) \in (0, M_k) \times \overline{B_0(1)}$:
\ben \label{propphik} 
 \phi_k(t,p) \in K_{k,t,p}^{\perp}  \quad \textrm{ and } \quad \Vert \phi_k(t,p) \Vert_{H^1} = O \left( \Vert R_{k,t,p} \Vert_{L^{\frac{2n}{n+2}}} \right),
\een
 where $K_{k,t,p}$ is as in \eqref{noyau} and where we have let
 \ben \label{erreur}
R_{k,t,p} = \left( \triangle_g + h \right)(u_0 + W_{k,t,p}) - f \left( u_0 + W_{k,t,p} \right)^{2^*-1} - \frac{\pi^2}{(u_0 + W_{k,t,p})^{2^*+1}} .
 \een 
 Note also that there holds $\eta(u_0 + W_{k,t,p}) = u_0 + W_{k,t,p}$ by the choice of $\eta$, hence the definition of $R_{k,t,p}$ makes sense. The proof of Proposition \ref{generalreduction} is in Robert-V\'etois \cite{RobertVetois}. Among the abundant literature, other possible references for the Lyapunov-Schmidt finite-dimensional reduction method are Ambrosetti-Malchiodi \cite{AmbrosettiMalchiodi}, Rey \cite{Rey}, Berti-Malchiodi \cite{BertiMalchiodi}, Del Pino-Musso-Pacard \cite{DelPinoMussoPacard}, Del Pino-Musso-Pacard-Pistoia \cite{DelPinoMussoPacardPistoia1, DelPinoMussoPacardPistoia2}, Lin-Ni-Wei \cite{LinNiWei}, Malchiodi-Ni-Wei \cite{MalchiodiNiWei}, Micheletti-Pistoia \cite{MichelettiPistoia}, Musso-Pacard-Wei \cite{MussoPacardWei} and Wei \cite{WeiGM}.

\section{$C^1$-estimates for the reduced energy} \label{energy}

\noindent Let $k \in \mathbb{N}$. We define the following auxiliary function, for $t >0$ and $p \in \overline{B_0(1)}$:
\ben \label{energiereduite}
I_k(t,p) = J(u_0 + W_{k,t,p}),
\een
where $J$ is as in \eqref{defJ} and $W_{k,t,p}$ is as in \eqref{bulle}. The main result of this section is a $C^1$-uniform asymptotic expansion of $I_k$.

\begin{prop} \label{propenergie}
Assume $n \ge 7$. There holds:
\begin{itemize}
\item If $(M,g)$ is locally conformally flat:
\be 
\bal
 I_k(t,p) - J(u_0) - \frac{1}{n}K_n^{-n} = - \frac{1}{2^*} \int_{\RR^n} \Psi(p+ty) \left( 1 + \frac{f(\xi_0)}{n(n-2)} |y|^2 \right)^{-n} dy \cdot \vek \\
- (n-2)^{\frac{n}{2}} n^{\frac{n-2}{2}} \delta_k(t)^{\frac{n-2}{2}} + o(\vek),
\eal
\ee
\item If $(M,g)$ is not locally conformally flat:
\be
\bal
 I_k(t,p) - J(u_0) - \frac{1}{n}K_n^{-n} = - \frac{1}{2^*} \int_{\RR^n} \Psi(p+ty) \left( 1 + \frac{f(\xi_0)}{n(n-2)} |y|^2 \right)^{-n} dy \cdot \vek \\
- (n-2)^{\frac{n}{2}} n^{\frac{n-2}{2}}\delta_k(t)^{\frac{n-2}{2}}  - K_n^{-n}\frac{n(n-2)^2}{24(n-4)(n-6)} |W(\xi_0)|_g^2 \delta_k(t)^4 +  o(\vek),
\eal
\ee
\end{itemize}
as $k \to \infty$, where $\Psi$ is as in \eqref{defPsi}. In addition, this expansion holds in $C^1_{loc}\left( (0, +\infty) \times \overline{B_0(1)} \right)$ as $k \to + \infty$. 
\end{prop}

\begin{proof}
It is easily seen, since $u_0$ solves \eqref{ELeps}, that there holds, for $(t,p) \in (0, +\infty) \times \overline{B_0(1)}$:
\ben \label{propener1}
\bal
J(u_0+W_{k,t,p}) & = J(u_0) + I_1 + I_2 - I_3,
\eal
\een
where we have let:
\ben \label{propener2}
\bal
I_1 &=  \frac{1}{2} \int_M \left(|\nabla W_{k,t,p}|^2 + h W_{k,t,p}^2 \right), \\
I_2 & =  \frac{1}{2^*} \int_M \pi^2 \left( (u_0+W_{k,t,p})^{-2^*} - u_0^{-2^*} + 2^* u_0^{-2^*-1} W_{k,t,p} \right)dv_g, \\
I_3 & = \frac{1}{2^*} \int_M f \left( (u_0 + W_{k,t,p})^{2^*} - u_0^{2^*} - 2^* u_0^{2^*-1} W_{k,t,p} \right)dv_g.
\eal
\een
We first compute $I_2$. The integral is localized on $B_{y_k}(2 r_k)$, the geodesic ball for the metric $g_{y_k}$, where $y_k$ is as in \eqref{defdkyk}. There holds, on the one hand:
\ben \label{propener3}
\bal
& \int_{B_{y_k}(\delta_k(t)^{\frac{1}{2}})}  \pi^2 \left( (u_0+W_{k,t,p})^{-2^*} - u_0^{-2^*} + 2^* u_0^{-2^*-1} W_{k,t,p} \right)dv_g \\
&= \int_{B_0(\delta_k(t)^{\frac{1}{2}})} \pi^2 \left( (u_0+W_{k,t,p})^{-2^*} - u_0^{-2^*} + 2^* u_0^{-2^*-1} W_{k,t,p} \right) \left( \Theta_k(y) \right) \\
& \quad \qquad \qquad \qquad \qquad \qquad \qquad \qquad \qquad  \times \left( \Lambda_{y_k}^{-2^*} \sqrt{|g_{y_k}|} \right)(\Theta_k(y))dy, \\
\eal
\een
where we have let, for $y \in B_0(2 r_k)$:
\ben \label{defThetak}
 \Theta_k(y) = \textrm{exp}_{y_k}^{g_{y_k}} (y),
\een
where $|g_{y_k}|$ denotes the quantity $\left| (\textrm{exp}_{y_k}^{g_{y_k}} )^* g_{y_k} \right|$ and where $y_k$ is as in \eqref{defdkyk}. As a first remark, \eqref{defdkyk} and standard smoothness arguments for the exponential map show that there always holds, for $y \in B_0(2 r_k)$:
 \ben \label{derThetak}
 \left| \frac{\partial}{\partial p_i} \Theta_k(y) \right| = O(\mk).
 \een
Also, the $C^1$ uniformity with respect to $\xi$ of the expansion \eqref{confnorm} shows that there holds, for fixed $y \in B_0(2 r_k)$ and for $1 \le i \le n$:
 \ben \label{propener8}
 \frac{\partial}{\partial p_i} \left( \sqrt{|g_{y_k}|}\right)(\Theta_k(y)) = O(|y|^N) 
 \een
 for $N$ large enough. And by the choice of $y_k$ in \eqref{defdkyk} and by \eqref{propLambda} there holds, for $y \in B_0(2 r_k)$:
\ben \label{propener7b}
  \left| \frac{\partial}{\partial p_i} \Lambda_{y_k}(y) \right| = O(\mk) \textrm{ and } \nabla \Lambda_{y_k} \left( \Theta_k(y) \right) = O(|y|).
\een
For $y \in B_0(\delta_k(t)^{\frac{1}{2}})$ we have that:
\ben \label{propener4}
 \left| (u_0+W_{k,t,p})^{-2^*} - u_0^{-2^*} + 2^* u_0^{-2^*-1} W_{k,t,p} \right| (\Theta_k(y)) = O \left( W_{k,t,p}(\Theta_k(y)) \right).
 \een
Let now $1 \le i \le n$. Using the computations given in \eqref{claimreste10} and \eqref{claimreste11} below, and since by \eqref{defZk} we have $|Z_{i,k,t,p}| = O(W_{k,t,p})$ and $|Z_{0,k,t,p}| = O(W_{k,t,p})$, we can write that:
\ben \label{propener5}
\bal
& \left| \frac{\partial}{\partial t } \left( (u_0+W_{k,t,p})^{-2^*} - u_0^{-2^*} + 2^* u_0^{-2^*-1} W_{k,t,p}\right) \right|(\Theta_k(y)) = O \left( W_{k,t,p} (\Theta_k(y)) \right) \textrm{ and }\\
& \left| \frac{\partial}{\partial p_i } \left( (u_0+W_{k,t,p})^{-2^*} - u_0^{-2^*} + 2^* u_0^{-2^*-1} W_{k,t,p}\right) \right|(\Theta_k(y)) = O \left( W_{k,t,p} (\Theta_k(y)) \right).
\eal
\een
It is easily seen with \eqref{condrk} and \eqref{bulle} that there holds in $M$:
\ben \label{contnablaW}
 |\nabla W_{k,t,p}| = O \left( \frac{1}{\delta_k} W_{k,t,p} \right),
 \een
so that there holds in turn, for $ 1 \le i \le n$ and for $|y| \le \delta_k(t)^{\frac{1}{2}}$:
 \ben \label{propener7}
   \bal
    \Bigg| \nabla &\left( (u_0+W_{k,t,p})^{-2^*} - u_0^{-2^*} + 2^* u_0^{-2^*-1} W_{k,t,p} \right) \Bigg|\left( \Theta_k(y) \right) \\
    & \qquad \qquad \qquad \qquad = O \left( \frac{1}{\delta_k} W_{k,t,p}(\Theta_k(y)) \right).
 \eal
 \een
In the end, gathering \eqref{derThetak}, \eqref{propener8},  \eqref{propener7b}, \eqref{propener4}, \eqref{propener5} and \eqref{propener7} we obtain with \eqref{propener3} that there holds 
\ben \label{propener9}
\int_{B_{y_k}(\delta_k(t)^{\frac{1}{2}})}  \pi^2 \left( (u_0+W_{k,t,p})^{-2^*} - u_0^{-2^*} + 2^* u_0^{-2^*-1} W_{k,t,p} \right)dv_g = O \left( \delta_k^{\frac{n}{2}}\right),
\een
and that this expansion holds in $C^1_{loc} ((0,+\infty) \times \overline{B_0(1)}$. On the other hand, using \eqref{bulle}, if $\delta_k(t)^{\frac{1}{2}} \le |y| \le 2r_k$ we have that 
\ben \label{propener10}
 \left| (u_0+W_{k,t,p})^{-2^*} - u_0^{-2^*} + 2^* u_0^{-2^*-1} W_{k,t,p} \right| (\Theta_k(y)) = O \left( W_{k,t,p}^2(\Theta_k(y)) \right).
\een
Using \eqref{bulle}, \eqref{claimreste10} and \eqref{claimreste11} below there also holds that:
\ben \label{propener11}
\bal
& \left| \frac{\partial}{\partial t } \left( (u_0+W_{k,t,p})^{-2^*} - u_0^{-2^*} + 2^* u_0^{-2^*-1} W_{k,t,p}\right) \right|(\Theta_k(y)) = O \left( W_{k,t,p} ^2(\Theta_k(y)) \right) \\ 
& \left| \frac{\partial}{\partial p_i } \left( (u_0+W_{k,t,p})^{-2^*} - u_0^{-2^*} + 2^* u_0^{-2^*-1} W_{k,t,p}\right) \right|(\Theta_k(y)) = O \left( W_{k,t,p}^2(\Theta_k(y)) \right),
\eal
\een
and using \eqref{contnablaW} we get that:
 \ben \label{propener12}
   \bal
     \Bigg| \nabla &\left( (u_0+W_{k,t,p})^{-2^*} - u_0^{-2^*} + 2^* u_0^{-2^*-1} W_{k,t,p} \right) \Bigg|\left( \Theta_k(y) \right) \\
    & \qquad \qquad \qquad \qquad = O \left( \frac{1}{\delta_k} W_{k,t,p}^2(\Theta_k(y)) \right).
  \eal
 \een
 Combining \eqref{propener10}, \eqref{propener11}, \eqref{propener12}, \eqref{derThetak}, \eqref{propener8} and \eqref{propener7b} gives in the end that
 \ben \label{propener13}
\bal
  \int_{B_{y_k}(2r_k) \backslash B_{y_k}(\delta_k(t)^{\frac{1}{2}})}  \pi^2 \Big( (u_0+W_{k,t,p})^{-2^*} - u_0^{-2^*} +  & 2^* u_0^{-2^*-1} W_{k,t,p} \Big) dv_g \\
  & = O \left( \delta_k^{\frac{n}{2}}\right)
\eal
 \een
 uniformly in $C^1_{loc} ((0,+\infty) \times \overline{B_0(1)})$. In the end, \eqref{propener9}, \eqref{propener13} and \eqref{defmk} therefore show that:
 \ben \label{propener14}
I_2 = o(\vek)
 \een
uniformly in  $C^1_{loc} ((0,+\infty) \times \overline{B_0(1)})$. To obtain \eqref{propener14} 
we used in particular that, by \eqref{defmk}, there holds for any $n \ge 7$ and for $t$ in compact subsets of $(0,+\infty)$, that 
\ben \label{controledeltak}
\left \{ \bal
\delta_k(t)^{\frac{n-2}{2}} & = O(\vek)  \textrm{ and } \\
\delta_k(t)^4 &= O(\vek) \textrm{ if } (M,g) \textrm{ is not l.c.f.}. \\
\eal \right.
\een
We now compute $I_1$ given by \eqref{propener2}. Using \eqref{coeffsEL} and \eqref{defPsi0} it is easily seen that there holds, uniformly in $C^0_{loc}( (0,+\infty) \times \overline{B_0(1)})$:
\ben \label{propener15}
\int_M \left( h - c_n S_g \right) W_{k,t,p}^2 dv_g = O (\vek^2).
\een
Using the conformal covariance property of the conformal laplacian it is easily seen that there holds, uniformly in compact subsets of $(0,+\infty) \times \overline{B_0(1)}$ (See for instance Esposito-Pistoia-V\'etois \cite{EspositoPistoiaVetois}):
\begin{itemize}
\item if $(M,g)$ is locally conformally flat:
\ben \label{propener16}
\frac{1}{2}\int_M \left( |\nabla W_{k,t,p}|^2 + c_n S_g W_{k,t,p}^2 \right)dv_g = \frac{1}{2} K_n^{-n} f(y_k)^{1- \frac{n}{2}} + O \left( \left( \frac{\delta_k}{r_k}\right)^{n-2} \right) 
\een
\item If $(M,g)$ is not locally conformally flat:
\ben \label{propener17}
\bal
\frac{1}{2}\int_M \left( |\nabla W_{k,t,p}|^2 + c_n S_g W_{k,t,p}^2 \right)dv_g & =  \frac{1}{2} K_n^{-n} f(y_k)^{1- \frac{n}{2}} + O \left( \left( \frac{\delta_k}{r_k}\right)^{n-2} \right) \\
- K_n^{-n} f(y_k)^{-1 - \frac{n}{2}} \frac{n(n-2)^2}{24(n-4)(n-6)}& |W(y_k)|_g^2 \delta_k(t)^4 + o(\delta_k(t)^4)  ,
\eal
\een
\end{itemize}
where in \eqref{propener16} and \eqref{propener17} we have let 
\ben \label{defKn}
K_n = \left( \frac{4}{n(n-2) \omega_n^{\frac{2}{n}}} \right)^{\frac{1}{2}} ,
\een
where $\omega_n$ is the volume of the standard $n$-sphere. That the expansions \eqref{propener15}, \eqref{propener16} and \eqref{propener17} also hold true uniformly in $C^1_{loc} ((0,+\infty) \times \overline{B_0(1)})$ follows from the same arguments than those developed in the computations from \eqref{propener3} to \eqref{propener12}. The explicit expression of $f$ in \eqref{coeffsEL},  the definition of $y_k$ in \eqref{defdkyk}, and \eqref{condrk}, \eqref{defmk} and \eqref{controledeltak} show, with \eqref{propener15}, \eqref{propener16} and \eqref{propener17}, that there holds in the end, uniformly in $C^1_{loc} ((0,+\infty) \times \overline{B_0(1)})$ :
\begin{itemize}
\item if $(M,g)$ is locally conformally flat:
\ben \label{calculI1lcf}
I_1 = \frac{1}{2} K_n^{-n} - \frac{n-2}{4} \Psi(p) \vek + o(\vek), 
\een
\item if $(M,g)$ is not locally conformally flat:
\ben \label{calculI1nonlcf}
I_1 = \frac{1}{2}K_n^{-n} - \frac{n-2}{4} \Psi(p) \vek - K_n^{-n}\frac{n(n-2)^2}{24(n-4)(n-6)} |W(\xi_0)|_g^2 \delta_k(t)^4 + o(\vek),
\een
\end{itemize}
where $\Psi$ is as in \eqref{defPsi} and $\delta_k(t)$ is as in \eqref{defdkyk}. It now remains to compute the integral $I_3$ in \eqref{propener2}. Here again, this integral is localized in $B_{y_k}(2 r_k)$. If $\sqrt{\delta_k(t)} \le d_{g_{y_k}}(y_k,y) \le 2 r_k$ we have that 
\ben \label{propener18}
 \left|  (u_0 + W_{k,t,p})^{2^*} - u_0^{2^*} - 2^* u_0^{2^*-1} W_{k,t,p} \right|(y) = O \left( W_{k,t,p}^2(y) \right).
 \een
Independently, if $d_{g_{y_k}}(y_k,y) \le \sqrt{\delta_k(t)}$, we have that
\ben \label{propener19}
\left| (u_0+W_{k,t,p})^{2^*} - W_{k,t,p}^{2^*} - 2^* u_0 W_{k,t,p}^{2^*-1} \right|(y) = O \left( W_{k,t,p}^{2^*-2}(y)\right).
\een
Using \eqref{propener18} and \eqref{propener19} we therefore get that there holds
\ben \label{propener20}
\bal
I_3 =   \frac{1}{2^*}\int_{B_{y_k}(\delta_k(t)^{\frac{1}{2}})} f W_{k,t,p}^{2^*} dv_g + \int_{B_{y_k}(\delta_k(t)^{\frac{1}{2}})} f u_0 W_{k,t,p}^{2^*-1} dv_g +O \left( \delta_k(t)^{\frac{n}{2}} \right). \\ 
\eal
\een
Using \eqref{coeffsEL}, straightforward computations give that: 
\ben \label{propener21}
\bal
& \int_{B_{y_k}(\delta_k(t)^{\frac{1}{2}})} f u_0 W_{k,t,p}^{2^*-1} dv_g  \\
& \qquad \qquad = (n-2)^{\frac{n}{2}} n^{\frac{n-2}{2}} \omega_{n-1} u_0(\xi_0) \delta_k(t)^{\frac{n-2}{2}} + o \left(\delta_k(t)^{\frac{n-2}{2}} \right) \\
& \qquad \qquad =  (n-2)^{\frac{n}{2}} n^{\frac{n-2}{2}} \omega_{n-1} \delta_k(t)^{\frac{n-2}{2}} + o \left(\delta_k(t)^{\frac{n-2}{2}} \right) \\
\eal
\een
since we chose $u_0 \equiv 1$ in \eqref{ELbase}. Independently, using \eqref{coeffsEL} there holds that
\ben \label{propener22}
\bal
&\int_{B_{y_k}(\delta_k(t)^{\frac{1}{2}})} f W_{k,t,p}^{2^*} dv_g = f(y_k)^{- \frac{n}{2}} K_n^{-n} +  o \left( \delta_k^{\frac{n-2}{2}} \right) \\ 
& \qquad + \vek \int_{B_0\left(\delta_k(t)^{- \frac{1}{2}} \right)} \Psi \left( \textrm{exp}_{y_k}^{g_{y_k}} (\delta_k(t) y) \right) \left(1 + \frac{f(y_k)}{n(n-2)} |y|^2 \right)^{-n}  dy. 
\eal
\een
Using the definition of $y_k$ and $\delta_k$ as in \eqref{defdkyk} it is easily seen that there holds:
\ben \label{propener23}
\frac{1}{\mk}\left( \exk \right)^{-1} \left( \textrm{exp}_{y_k}^{g_{y_k}} (\delta_k(t) y) \right) \to p + ty
\een
in $C^1_{loc}(\RR^n)$, as $k \to + \infty$. Therefore, with \eqref{propener23}, Lebesgue's dominated convergence theorem and \eqref{coeffsEL}, equation \eqref{propener22} becomes:
\ben \label{propener24}
\bal
& \int_{B_{y_k}(\delta_k(t)^{\frac{1}{2}})} f W_{k,t,p}^{2^*} dv_g = K_n^{-n} - \frac{n}{2} \Psi(p) \vek \\
& + \int_{\RR^n} \Psi(p + t y) \left( 1 + \frac{f(\xi_0)}{n(n-2)} |y|^2\right) dy \cdot \vek + o(\vek) + o \left(  \delta_k(t)^{\frac{n-2}{2}} \right).
\eal
\een
Using again \eqref{controledeltak} and \eqref{condrk}, \eqref{propener21} and \eqref{propener24} yield in the end in \eqref{propener20} that:
\ben \label{propener25}
\bal
I_3 &= \frac{1}{2^*}K_n^{-n} - \frac{n-2}{4} \Psi(p) \vek + (n-2)^{\frac{n}{2}} n^{\frac{n-2}{2}} \omega_{n-1} \delta_k(t)^{\frac{n-2}{2}}\\
&+ \frac{1}{2^*}\int_{\RR^n} \Psi(p + t y) \left( 1 + \frac{f(\xi_0)}{n(n-2)} |y|^2\right) dy \cdot \vek + o(\vek),
\eal
\een
where $\Psi$ is as in \eqref{defPsi}. That expansion \eqref{propener25} holds here also in $C^1_{loc}((0,+\infty)\times \overline{B_0(1)})$ is again a consequence of the arguments developed in \eqref{propener3}--\eqref{propener12}. Gathering \eqref{propener14}, \eqref{calculI1lcf}, \eqref{calculI1nonlcf} and \eqref{propener25} in \eqref{propener1} 
therefore concludes the proof of Proposition \ref{propenergie}.
\end{proof}

\section{Error estimates} \label{reste}

\noindent Let $\phi_k$ be the function given by Proposition \ref{generalreduction}. The following proposition controls the error terms due to $\phi_k$ in the expansion of the energy.

\begin{prop} \label{propreste}
There holds:
\ben \label{DLreste}
J \Big( u_0 + W_{k,t,p} +  \phi_k(t,p) \Big) = I_k(t,p) + o(\vek)
\een
as $k \to \infty$, where $I_k(t,p)$ is defined in \eqref{energiereduite} and $J$ is defined in \eqref{defJ}. Furthermore, this expansion holds in $C^1_{loc}\left( (0, +\infty) \times \overline{B_0(1)} \right)$.
\end{prop}

\begin{proof}
Since $u_0$ is a solution of equation \eqref{ELeps}, straightforward computations using the expression of $J$ as in \eqref{defJ} and \eqref{propphik} show that there holds for any $(t,p) \in (0, +\infty) \times \overline{B_0(1)}$: 
\ben \label{correctionerreur}
J \left(u_0 + W_{k,t,p} + \phi_k(t,p) \right) = J(u_0 + W_{k,t,p}) + O \left( \Vert R_{k,t,p} \Vert_{L^{\frac{2n}{n-2}}}^2 \right),
\een
whrere $R_{k,t,p}$ is as in \eqref{erreur}. To compute $R_{k,t,p}$ we write that, since $u_0$ solves \eqref{ELeps}:
\ben \label{calcerr1}
\bal
R_{k,t,p} & = 
 f \left( u_0^{2^*-1} + W_{k,t,p}^{2^*-1} - (u_0+W_{k,t,p})^{2^*-1}\right)  + \left( \triangle_g + c_n S_g \right) W_{k,t,p} - f(y_k) W_{k,t,p}^{2^*-1} \\
& + \pi^2 \left( u_0^{-2^*-1} - (u_0 + W_{k,t,p})^{-2^*-1} \right) + \left( f(y_k) - f \right) W_{k,t,p}^{2^*-1} \\
& + \left( h - c_n S_g\right) W_{k,t,p} .
\eal
\een
By its definition in \eqref{bulle}, $W_{k,t,p}$ is supported in $B_{y_k}(2r_k)$ (the geodesic ball is taken with respect to the metric $g_{y_k}$). Hence there holds, using \eqref{coeffsEL}, that:
\ben \label{calcerr2}
\left| u_0^{2^*-1} + W_{k,t,p}^{2^*-1} - (u_0+W_{k,t,p})^{2^*-1} \right| = O \left( \min \left( W_{k,t,p}, W_{k,t,p}^{2^*-2} \right)\right),
\een
that
\ben \label{calcerr2b}
\left|  \pi^2 \left( u_0^{-2^*-1} - (u_0 + W_{k,t,p})^{-2^*-1} \right)  \right| = O \Big( \min \left(1, W_{k,t,p} \right) \Big),
\een
that
\ben \label{calcerr3}
\left|  f(y_k) - f  \right| W_{k,t,p}^{2^*-1} = O \left( \frac{\vek}{\mk} d_{g_{y_k}}(y_k,\cdot) W_{k,t,p}^{2^*-1} \right),
\een
and that
\ben \label{calcerr4}
 \left| h - c_n S_g \right| W_{k,t,p}  = O \left( \left( \frac{\vek}{\mk}\right)^2 W_{k,t,p} \right). 
\een
Because of the conformal invariance property of the conformal laplacian, by letting $\tilde{W}_k = \Lambda_{y_k}^{-1} W_{k,t,p}$ there holds:
\ben \label{confcovlap}
\left( \triangle_g + c_n S_g\right)W_{k,t,p}(x) = \Lambda_{y_k}^{2^*-1} \left( \triangle_{g_{y_k}} + c_n S_{g_{y_k}} \right)\tilde{W}_k.
\een
Since the function $\tilde{W}_k$ is radial in normal coordinates for the metric $g_{y_k}$, by \eqref{confnorm} and \eqref{condrk}, straightforward computations show that there holds:
\ben \label{calcerr5}
\triangle_{g_{y_k}} \tilde{W}_k (x) = f(y_k) \tilde{W}_k^{2^*-1} + O \left( \mk^{\frac{n-2}{2}}r_k^{-n} \textbf{1}_{r \ge r_k}\right) + O(\mk^{\frac{n-2}{2}} ),
\een
where we have let $r = d_{g_{y_k}}(y_k,x)$. If $(M,g)$ is locally conformally flat around $y_k$ then there holds $S_{g_{y_k}} \equiv 0$ in a neighborhood of $y_k$. If $(M,g)$ is not locally conformally flat, letting $r = d_{g_{y_k}}(y_k,x)$, 
there holds $S_{g_{y_k}} = O (r^2)$, so that:
\ben \label{calcerr6}
S_{g_{y_k}} W_{k,t,p} = O \left( r^2 W_{k,t,p}\right).
\een
Using estimates \eqref{calcerr2} to \eqref{calcerr6} and using \eqref{defdkyk}, straightforward computations show that:
\ben \label{calcerr7}
 \Vert R_{k,t,p} \Vert_{L^{\frac{2n}{n+2}}}^2 = O(\delta_k^{\frac{n+2}{2}}) + \alpha_{n,k} + O (\vek^2 ) + O \left( \left( \frac{\delta_k}{r_k} \right)^{n-2}\right)
 \een
for any $n \ge 7$, 
where we have let: 
\ben \label{defalphak}
 \alpha_{n,k} = \left \{
\bal
& 0 &\textrm{ if } (M,g) \textrm{ is l.c.f.}&\\
& O \left( \left( \frac{\delta_k}{r_k}\right)^{n-2} \right)  &\textrm{ if } 7 \le n \le 9 \textrm{ and } (M,g) \textrm{ is not l.c.f.}&\\ 
& O \left(\mk^8 \left| \ln \delta_k \right|^{\frac{6}{5}} \right)   &\textrm{ if } n = 10 \textrm{ and } (M,g) \textrm{ is not l.c.f.}&\\ 
& O(\delta_k^8)  &\textrm{ if } n \ge 11 \textrm{ and } (M,g) \textrm{ is not l.c.f.}& \\
\eal
\right.
\een
The expansion in \eqref{calcerr7}  
is uniform with respect to $(t,p)$ in compact subsets of $(0, +\infty) \times \overline{B_0(1)}$. With \eqref{defmk}, \eqref{defdkyk} and \eqref{controledeltak} we therefore obtain that there holds,
 for any $n \ge 7$ and uniformly in $(t,p) \in (0, +\infty) \times \overline{B_0(1)}$:
\ben \label{calcerr9}
\Vert R_{k,t,p} \Vert_{L^{\frac{2n}{n+2}}}^2 = o(\vek).
\een
With \eqref{correctionerreur}, \eqref{calcerr9} therefore proves the expansion \eqref{DLreste} uniformly in $C^0_{loc}( (0, +\infty) \times \overline{B_0(1)})$.

\medskip

\noindent We now prove that \eqref{DLreste} actually holds in $C^1_{loc}( (0, +\infty) \times \overline{B_0(1)})$. 
The proof goes through a series of Claims.

\begin{claim} \label{claimreste1}
For any $(t,p) \in (0, + \infty) \times \overline{B_0(1)}$ and for any $1 \le i \le n$ there holds:
\ben \label{eqclaim1}
\bal
\frac{\partial}{\partial t} \Big( J(u_0 + W_{k, t, p}) \Big)(t, p) & = \frac{n-2}{2t} DJ(u_0 + W_{k,t,p}) \left( Z_{0,k,t,p} \right) + o(\vek)\\
\frac{\partial}{\partial p_i}  \Big( J(u_0 + W_{k, t, p}) \Big)(t, p) & =  \frac{1}{n t} DJ(u_0 + W_{k,t,p}) \left(Z_{i,k,t,p} \right) + o(\vek) .\\
\eal
\een
As a consequence there holds, for any $0 \le i \le n$:
\be
DJ(u_0 + W_{k,t,p}) \left(Z_{i,k,t,p} \right) = O(\vek).
\ee
\end{claim}
\begin{proof}[Proof of Claim \ref{claimreste1}]
It is easily seen that there holds, for any $x \in M$:
\ben \label{claimreste10}
\frac{\partial}{\partial t} W_{k,t,p}(x) = \frac{n-2}{2t} Z_{0,k,t,p}(x),
\een
and therefore that the first expansion in Claim \ref{claimreste1} holds true. We now turn to the derivatives in $p$. Straightforward computations using \eqref{bulle} and the properties of the conformal factor $\Lambda$ given in \eqref{propLambda} show that there holds, for any $x \in M$ and $1 \le i \le n$:
\ben \label{claimreste11}
\frac{\partial}{\partial p_i} W_{k,t,p}(x) = \frac{1}{n t} Z_{i,k,t,p}(x) + O \left( \delta_k W_{k,t,p} \right) + O \left( \frac{\delta_k}{r_k} \chi' \left( \frac{d_{g_{y_k}}(y_k,x)}{r_k}\right) W_{k,t,p} \right).
\een
We have, thanks to \eqref{claimreste11} and to \eqref{defJ}, that:
\ben \label{claimreste12}
\bal
  &\frac{\partial}{\partial p_i}  \Big( J(u_0 + W_{k, t, p}) \Big)(t, p) -  \frac{1}{n t} DJ(u_0 + W_{k,t,p})  \left( Z_{j,k,t,p} \right) \\
 & = O \left( \int_M \delta_k W_{k,t,p} |R_{k,t,p}| dv_g \right) + O  \left( \int_M \frac{\delta_k}{r_k} \chi' \left( \frac{d_{g_{y_k}}(y_k,x)}{r_k}\right) W_{k,t,p} |R_{k,t,p}| dv_g\right),
 \eal
 \een
 where $R_{k,t,p}$ is as in \eqref{erreur}. 
Using the pointwise control on $R_{k,t,p}$ given by \eqref{calcerr1} and by the estimates \eqref{calcerr2} to \eqref{calcerr6} one gets that \eqref{claimreste12} rewrites as 
\ben \label{claimreste14}
 \frac{\partial}{\partial p_i}  \Big( J(u_0 + W_{k, t, p}) \Big)(t, p) -  \frac{1}{n t} DJ(u_0 + W_{k,t,p}) \left(Z_{i,k,t,p} \right) = o(\vek).
\een
In the end, \eqref{claimreste10} and \eqref{claimreste14} prove \eqref{eqclaim1}. Independently, the $C^1$-uniform expansion of the reduced energy given by Proposition \ref{propenergie} shows, using \eqref{controledeltak}, that for any $1 \le i \le n$ there holds:
\ben \label{claimreste15}
\bal
 \frac{\partial}{\partial t}  \Big( J(u_0 + W_{k, t, p}) \Big)(t, p) & = O(\vek) \textrm{ and } \\
  \frac{\partial}{\partial p_i}  \Big( J(u_0 + W_{k, t, p}) \Big)(t, p) & = O(\vek) \\
\eal
\een
as $k \to + \infty$. 
Combining this with \eqref{eqclaim1} concludes the proof of Claim \ref{claimreste1}.
\end{proof}

\begin{claim} \label{claimreste2}
Let $\lambda_{i,k}$ be the functions defined by Proposition \ref{generalreduction}. There holds, for $(t,p) \in (0, +\infty) \times \overline{B_0(1)}$ and for $1 \le i \le n$:
\ben
\bal
\frac{\partial}{\partial t} \Big( J(u_{k,t,p}) \Big)(t, p) & = \frac{n-2}{2t}  \Vert \nabla V_{0,k}\Vert_{L^2(\RR^n)}^2 \lambda_{0,k}(t,p) + o(S_k),\\
\frac{\partial}{\partial p_i}  \Big( J(u_{k,t,p} ) \Big)(t, p) & =  \frac{1}{n t} \Vert  \nabla V_{i,k} \Vert_{L^2(\RR^n)}^2 \lambda_{i,k}(t,p) + o(S_k),\\
\eal
\een
where $u_{k,t,p}$ is as in \eqref{defuk}, the $V_{i,k}$, $0 \le i \le n$ are as in \eqref{defVki} and where we have let:
\ben \label{defSk}
S_k = \sum_{i=0}^n |\lambda_{i,k}(t,p)|.
\een
\end{claim}

\begin{proof}[Proof of Claim \ref{claimreste2}]
Using equation \eqref{soloutnoyau} it is easily seen that there hold:
\ben \label{claimreste21}
\bal
& \frac{\partial}{\partial t} \Big( J(u_{k,t,p}) \Big)(t, p) \\
& = \left \langle  \lambda_{0,k}(t,p) Z_{0,k,t,p} + \sum_{j = 1}^n \lambda_{j,k}(t,p) Z_{j,k,t,p}, \frac{\partial W_{k,t,p}}{\partial t} + \frac{\partial \phi_k(t,p)}{\partial t}\right \rangle_{h}
\eal
\een
and, for $1 \le i \le n$:
\ben \label{claimreste21b}
\bal
& \frac{\partial}{\partial p_i} \Big( J(u_{k,t,p}) \Big)(t, p) \\
& = \left \langle  \lambda_{0,k}(t,p) Z_{0,k,t,p} + \sum_{j = 1}^n \lambda_{j,k}(t,p) Z_{j,k,t,p}, \frac{\partial W_{k,t,p}}{\partial p_i} + \frac{\partial \phi_k(t,p)}{\partial p_i}\right \rangle_{h}.
\eal
\een
Using the fact that $\phi_k(t,p) \in K_{k,t,p}^{\perp}$ (see \eqref{propphik}) we get that for any $0 \le j \le n$ and any $1 \le i \le n$ there holds:
\ben \label{claimreste21c}
\bal
\left \langle  Z_{j,k,t,p},  \frac{\partial \phi_k(t,p)}{\partial t} \right \rangle_{h} & = - \left \langle \frac{\partial Z_{j,k,t,p}}{\partial t}, \phi_k(t,p) \right \rangle_{h} \textrm{ and } \\
\left \langle  Z_{j,k,t,p},  \frac{\partial \phi_k(t,p)}{\partial p_i} \right \rangle_{h} & = - \left \langle \frac{\partial Z_{j,k,t,p}}{\partial p_i}, \phi_k(t,p) \right \rangle_{h}.
\eal
\een
Straightforward computations using \eqref{defZk} (see for instance Robert-V\'etois \cite{RobertVetois}) show that there holds, for $0 \le j \le n$ and $1 \le i \le n$:
\ben \label{claimreste22}
 \left \Vert  \frac{\partial Z_{j,k,t,p}}{\partial t} \right \Vert_{H^1(M)} +  \left \Vert  \frac{\partial Z_{j,k,t,p}}{\partial p_i} \right \Vert_{H^1(M)} = O(1),
 \een
 and this expression is uniform in $t$ and $p$ in compact subsets of $(0, + \infty) \times \overline{B_0(1)}$. By Proposition \ref{generalreduction} there holds $\Vert \phi_{k,t,p} \Vert_{H^1(M)} = O (\Vert R_{k,t,p} \Vert_{L^{\frac{2n}{n+2}}})$, where $R_{k,t,p}$ is as in \eqref{erreur}, so that with \eqref{calcerr7} and \eqref{claimreste22} equation \eqref{claimreste21c} gives:
\ben \label{claimreste23}
\bal
\left \langle  Z_{j,k,t,p},  \frac{\partial \phi_k(t,p)}{\partial t} \right \rangle_{h} & = o(1) \textrm{ and } \left \langle  Z_{j,k,t,p},  \frac{\partial \phi_k(t,p)}{\partial p_i} \right \rangle_{h} & = o(1) 
\eal
\een
as $k \to + \infty$. Using the definition of the $Z_{i,k,t,p}$ as in \eqref{defZk} it is easily seen that there holds for $0 \le i,j \le n$:
\ben \label{claimreste24}
\left \langle Z_{i,k,t,p}, Z_{j,k,t,p} \right\rangle_{h} = \delta_{ij} \Vert \nabla V_{i,k}\Vert_{L^2(\RR^n)}^2 + o(1),
\een
as $k \to + \infty$, where $V_{i,k}$ is as in \eqref{defVki}. Combining \eqref{claimreste10} and \eqref{claimreste11} with \eqref{claimreste24} we therefore obtain that for any $0 \le j \le n$ and $1 \le i \le n$ there holds:
\ben \label{claimreste25}
\bal
\left \langle Z_{j,k,t,p}, \frac{\partial W_{k,t,p}}{\partial t} \right \rangle_h & = \delta_{0j} \frac{n-2}{2t}\Vert \nabla V_{0,k} \Vert_{L^2(\RR^n)}^2 + o(1) \textrm{ and } \\
\left \langle Z_{j,k,t,p}, \frac{\partial W_{k,t,p}}{\partial p_i} \right \rangle_h & = \delta_{ij} \frac{1}{nt} \Vert \nabla V_{i,k} \Vert_{L^2(\RR^n)}^2 + o(1) \\
\eal
\een
as $k \to + \infty$. Claim \ref{claimreste2} follows from combining \eqref{claimreste23} and \eqref{claimreste25} with \eqref{claimreste21} and \eqref{claimreste21b}.
\end{proof}

\begin{claim} \label{claimreste3}
There holds, for any $0 \le i \le n$ and for any $(t,p) \in (0, +\infty) \times \overline{B_0(1)}$:
\be
DJ(u_{k,t,p}) \Big( Z_{i,k,t,p} \Big) = DJ ( u_0 + W_{k,t,p}) \Big( Z_{i,k,t,p}\Big) + o(\vek)
\ee 
as $k \to + \infty$, where $u_{k,t,p}$ is as in \eqref{defuk}, and this expansion is uniform with respect to $t$ and $p$ in $(0, +\infty) \times \overline{B_0(1)}$.
\end{claim}

\begin{proof}[Proof of Claim \ref{claimreste3}]
By the definition of $J$ as in \eqref{defJ} and since $n \ge 7$ there holds that for $0 \le i \le n$:
\ben \label{claimreste31}
\bal
& \left ( DJ(u_{k,t,p}) -  DJ ( u_0 + W_{k,t,p}) \right) \Big( Z_{i,k,t,p}\Big) =   \left \langle \phi_k(t,p), L_k Z_{i,k,t,p} \right \rangle_h \\
& + O \left( \int_M (u_0 + W_{k,t,p})^{2^*-3}\left| Z_{i,k,t,p} \right| \left|\phi_k(t,p) \right|^2 dv_g \right) \\
& +  O \left( \int_M \left| W_{k,t,p}^{2^*-2} - (u_0+W_{k,t,p})^{2^*-2} \right| \left| Z_{i,k,t,p} \right|  \left| \phi_k(t,p)\right| dv_g \right) \\
& + O \left( \int_M \left|\eta \left(u_0 + W_{k,t,p} + \phi_k(t,p) \right)^{-2^*-1} - \left( u_0 + W_{k,t,p} \right)^{-2^*-1} \right| \left| Z_{i,k,t,p}\right|dv_g\right),
\eal
\een
where we have let, for any $u \in H^1(M)$:
\ben \label{defLk}
L_k u = u  - \left( \triangle_g + h \right)^{-1} \left( (2^*-1)f W_{k,t,p}^{2^*-2} u \right).
\een
There holds $|Z_{j,k,t,p}| = O(W_{k,t,p})$ for all $0 \le j \le n$, so that a H\"{o}lder inequality, \eqref{propphik} and \eqref{calcerr9} give that:
\ben \label{claimplus1}
\int_M (u_0 + W_{k,t,p})^{2^*-3}\left| Z_{i,k,t,p} \right| \left|\phi_k(t,p) \right|^2 dv_g = o (\vek).
\een
Similarly, we can write that 
\ben \label{claimplus2}
\bal
 & \int_M \left| W_{k,t,p}^{2^*-2} - (u_0+W_{k,t,p})^{2^*-2} \right| \left| Z_{i,k,t,p} \right|  \left| \phi_k(t,p)\right| dv_g \\
 & \quad = O \left( \int_M \left| \theta_k  \right| \left| \phi_k(t,p) \right| dv_g \right),
\eal
\een
where $\theta_k$ is a function satisfying:
\ben \label{claimplus2a}
\left| \theta_k(y) \right| =  \left \{
\bal
& O( W_{k,t,p}^{2^*-2}) &\textrm{ if } d_{g_{y_k}}(y_k,y) \le \sqrt{\delta_k(t)}& \\
& O(W_{k,t,p}) &\textrm{ if } d_{g_{y_k}}(y_k,y) \ge \sqrt{\delta_k(t)} .& \\
\eal \right.
\een
Using \eqref{claimplus2a} in \eqref{claimplus2}, a H\"{o}lder inequality gives then, with \eqref{propphik}, \eqref{calcerr7}, \eqref{controledeltak} and \eqref{condrk}, that:
\ben \label{claimplus22}
\int_M \left| W_{k,t,p}^{2^*-2} - (u_0+W_{k,t,p})^{2^*-2} \right| \left| Z_{i,k,t,p} \right|  \left| \phi_k(t,p)\right| dv_g  = o(\vek).
\een
Finally, we write that:
\ben \label{claimplus3}
\bal
\left|Z_{i,k,t,p}\right|(y) \left |\eta \left(u_0 + W_{k,t,p} + \phi_k(t,p) \right)^{-2^*-1} - \left( u_0 + W_{k,t,p} \right)^{-2^*-1} \right|(y) \\
= \left \{
\bal
& O(\left|Z_{i,k,t,p}\right|) &\textrm{ if }  d_{g_{y_k}}(y_k,y) \le \sqrt{\delta_k(t)}& \\
& O \Big( \left| \phi_k(t,p)\right|(y) \left| Z_{i,k,t,p}\right|(y) \Big) &\textrm{ if } d_{g_{y_k}}(y_k,y) \ge \sqrt{\delta_k(t)}.& \\
\eal \right.
\eal
\een
A H\"{o}lder inequality using \eqref{claimplus3}, \eqref{defZk}, \eqref{propphik}, \eqref{calcerr7}, \eqref{controledeltak} and \eqref{condrk} gives then:
\ben \label{claimplus33}
 \int_M \left|\eta \left(u_0 + W_{k,t,p} + \phi_k(t,p) \right)^{-2^*-1} - \left( u_0 + W_{k,t,p} \right)^{-2^*-1} \right| \left| Z_{i,k,t,p}\right|dv_g = o(\vek).
\een
With \eqref{claimplus1}, \eqref{claimplus22} and \eqref{claimplus33}, equation \eqref{claimreste31} becomes then:
\ben \label{claimreste32}
\bal
\left ( DJ(u_{k,t,p}) -  DJ ( u_0 + W_{k,t,p}) \right) &\Big( Z_{i,k,t,p}\Big) =  o(\vek) \\
& + O \left( \Vert \phi_k(t,p) \Vert_{H^1(M)} \Vert L_k Z_{i,k,t,p} \Vert_{H^1(M)} \right) \\
\eal
\een
as $k \to + \infty$, uniformly in $t$ and $p$ in compact subsets of $(0, + \infty) \times \overline{B_0(1)}$. We now compute $\Vert L_k Z_{i,k,t,p} \Vert_{H^1(M)} $. There holds:
\ben \label{claimreste33}
\Vert L_k Z_{i,k,t,p} \Vert_{H^1(M)} = O \left( \left \Vert (\triangle_g +h) Z_{i,k,t,p} - (2^*-1)f W_{k,t,p}^{2^*-2} Z_{i,k,t,p} \right \Vert_{L^{\frac{2n}{n+2}}} \right).
\een
Using the conformal covariance property of the conformal laplacian \eqref{confcovlap} we get, using the expressions \eqref{coeffsEL} for the coefficients of \eqref{ELeps}, that
\be 
\bal
& (\triangle_g +h) Z_{i,k,t,p}  - (2^*-1)f W_{k,t,p}^{2^*-2} Z_{i,k,t,p}  = (h - c_n S_g) Z_{i,k,t,p} \\
& + (2^*-1) \left( f(y_k) - f \right) W_{k,t,p}^{2^*-2} Z_{i,k,t,p} + O \left( \mk^{\frac{n-2}{2}}r_k^{-n} \textbf{1}_{r \ge r_k} \right) \\
&+ \left \{
\bal
& 0 &\textrm{ if } (M,g) \textrm{ is l.c.f.} &\\
& O \left( r^2 Z_{i,k,t,p}\right) &\textrm{ otherwise} &\\
\eal \right.
\eal
\ee
where we have let $r = d_{g_{y_k}}(y_k, \cdot)$. Since there holds $|Z_{i,k,t,p}| = O(W_{k,t,p})$, equations \eqref{calcerr3} and \eqref{calcerr4} tell us with \eqref{claimreste33} that there holds:
\ben \label{claimreste35}
\Vert L_k Z_{i,k,t,p} \Vert_{H^1(M)} = O(\vek) + O \left( \left( \frac{\mk}{r_k} \right)^{\frac{n-2}{2}} \right) + \alpha_k^{\frac{1}{2}},
\een
where $\alpha_k$ is as in \eqref{defalphak}. Using \eqref{propphik}, \eqref{calcerr7} and \eqref{claimreste35} we then obtain that 
\[  \Vert \phi_k(t,p) \Vert_{H^1(M)} \Vert L_k Z_{i,k,t,p} \Vert_{H^1(M)} = o(\vek) \]
as $k \to + \infty$. With \eqref{claimreste32} this concludes the proof of Claim \ref{claimreste3}.
\end{proof}

We are now in position to finish the proof of Proposition \ref{DLreste}. First, note that as a consequence of \eqref{soloutnoyau} and of \eqref{claimreste24} we can write that there holds, for any $1 \le j \le n$:
\ben \label{conclureste1}
\bal
DJ(u_{k,t,p}) \Big( Z_{0,k,t,p} \Big) & =\Vert \nabla V_{0,k} \Vert_{L^2(\RR^n)}^2 \lambda_{0,k}(t,p) + o(S_k) \textrm{ and } \\
DJ(u_{k,t,p}) \Big( Z_{j,k,t,p} \Big) & = \Vert \nabla V_{j,k} \Vert_{L^2(\RR^n)}^2 \lambda_{i,k}(t,p) + o(S_k) \\
\eal
\een
for $k \to +\infty$, where $S_k$ is as in \eqref{defSk} and $u_{k,t,p}$ is as in \eqref{defuk}. As a consequence of Claim \ref{claimreste3} and of Claim \ref{claimreste1} we can then write that:
\ben \label{conclureste2}
S_k = O(\vek).
\een
Using successively Claim \ref{claimreste2}, \eqref{conclureste1}, Claim \ref{claimreste3} and Claim \ref{claimreste1} we therefore obtain, for $1 \le i \le n$:
\ben \label{conclureste3}
\bal
 \frac{\partial}{\partial t} \Big( J(u_{k,t,p}) \Big)(t, p) & = \frac{\partial}{\partial t} \Big( J(u_0 + W_{k,t,p}) \Big)(t, p) + o(\vek) \textrm{ and }\\
 \frac{\partial}{\partial p_i}  \Big( J(u_{k,t,p} ) \Big)(t, p) & =  \frac{\partial}{\partial p_i}  \Big( J(u_0 + W_{k,t,p} ) \Big)(t, p) + o(\vek).\\
\eal
\een
With \eqref{calcerr9} and \eqref{correctionerreur}, \eqref{conclureste3} concludes the proof of Proposition \ref{DLreste}. 
\end{proof}

\section{Conclusion of the proof of Theorem \ref{Th1}} \label{conclusion}

\noindent We conclude in this section the proof of Theorem \ref{Th1} by showing that the $\lambda_{i,k}$ as in \eqref{soloutnoyau} can be made to vanish for some values of $t$ and $p$. 

\medskip

\noindent Assume first that $7 \le n \le 9$ or that $(M,g)$ is locally conformally flat. Then by \eqref{defmk} we have $\mk = \vek^{\frac{2}{n-2}}$ and  Propositions \ref{propenergie} and \ref{propreste} show that there holds:
\ben \label{conclu1}
\frac{1}{\vek} \left( J \Big( u_0 + W_{k,t,p} + \phi_k(t,p) \Big) - J(u_0) - \frac{1}{n}K_n^{-n} \right) \to H(t,p)
\een
in $C^1_{loc} \left( (0,+\infty) \times \overline{B_0(1)} \right)$ as $k \to + \infty$, where we have let:
\ben \label{defH}
H(t,p) = - \frac{1}{2^*} \int_{\RR^n} \Psi(p+ty) \left( 1 + \frac{f(\xi_0)}{n(n-2)} |y|^2 \right)^{-n} dy 
- \alpha_n t^{\frac{n-2}{2}}. 
\een
In \eqref{defH} we have let $\alpha_n = (n-2)^{\frac{n}{2}} n^{\frac{n-2}{2}} $ and $\Psi$ is as in \eqref{defPsi}. For $t > 0$, straightforward calculations show that there holds
\ben \label{conclu2}
\bal
\partial_t H(t,0)  &= \frac{2}{2^*} t \int_{\RR^n} |y|^2 \beta(t^2 |y|^2) \left( 1 + \frac{f(\xi_0)}{n(n-2)} |y|^2 \right)^{-n}dy - \frac{n-2}{2} \alpha_n  t^{\frac{n-4}{2}} \\
& + \frac{2}{2^*} t^3 \int_{\RR^n} |y|^4 \beta'(t^2 |y|^2)  \left( 1 + \frac{f(\xi_0)}{n(n-2)} |y|^2 \right)^{-n}dy.
\eal
\een
We let $t_0 >0$ be defined by:
\ben \label{deft0}
 \frac{n-2}{2} \alpha_n  {t_0}^{\frac{n-6}{2}} = \frac{2}{2^*} \int_{\RR^n} |y|^2 \left( 1 + \frac{f(\xi_0)}{n(n-2)} |y|^2 \right)^{-n}dy.
\een
In \eqref{defPsi} we chose $\beta$ to be compactly supported in $[-M-1, M+1]$ and to be equal to $1$ in $[-M,M]$, for some positive $M$. With \eqref{conclu2} it is therefore not difficult to see that, up to choosing $M$ large enough, there exists $t_M > 0$ such that 
\ben \label{conclu3}
\partial_t H(t_M, 0 ) = 0,
\een 
and such that 
\ben \label{asymptotm}
t_M = t_0 + o_M(1)
\een
 as $M \to +\infty$, where $o_M(1)$ denotes some quantity that goes to zero as $M \to \infty$. We then have that:
\[ \partial_t^2 H(t_M,0) = - \frac{n-2}{2} \frac{n-6}{2} \alpha_n t_0^{\frac{n-6}{2}} + o_M(1),\]
so that up to choosing  a fixed $M$ large enough there holds:
\ben \label{conclu4}
\partial_t^2 H(t_M,0)  < 0.
\een
Let now $1 \le i,j \le n$. With \eqref{defH} there holds, for any $t >0$:
\ben \label{conclu5}
\bal
\partial_i H(t,0) & = - \frac{1}{2^*} \int_{\RR^n} \partial_i \Psi(ty) \left( 1 + \frac{f(\xi_0)}{n(n-2)} |y|^2 \right)^{-n}dy \\
 & = 0 \\
\eal
\een
since $\Psi$ in \eqref{defPsi} is radial. Using \eqref{defH} it is also easily seen that there holds:
\ben \label{conclu6}
 \partial_i \partial_j H(t_M,0) = \frac{2}{2^*} K_n^{-n}\delta_{ij} + o_M(1), 
 \een
where $K_n$ is as in \eqref{defKn}. Finally, using \eqref{conclu5} one gets that:
\ben \label{conclu7}
\partial_t \partial_i H(t_M,0) = 0.
\een
In the end, equations \eqref{conclu3} to \eqref{conclu7} show that, up to choosing $M$ large enough, $(t_M, 0)$ is a non-degenerate critical point of $H$. We claim now that this is enough to conclude the proof of Theorem \ref{Th1} when $7 \le n \le 9$ or $(M,g)$ is locally conformally flat. Define indeed, for $(t,p) \in (0,+\infty)\times \overline{B_0(1)}$ and for any $k$:
\ben \label{conclu8}
H_k(t,p) = \frac{1}{\vek} \left( J \Big( u_0 + W_{k,t,p} + \phi_k(t,p) \Big) - J(u_0) - \frac{1}{n}K_n^{-n} \right).
\een
Since $(t_M,0)$ is a nondegenerate critical point of $H$ in \eqref{defH}, and by \eqref{asymptotm}, there exists $\gamma >0$ such that $(t_M,0)$ is the only zero of $\nabla_{(t,p)}H$ in $[t_0 - \gamma, t_0 + \gamma] \times \overline{B_0(\gamma)}$. The $C^1$-convergence  \eqref{conclu1} shows then that for $k$ large enough and for any $s \in [0,1]$,
\[\textrm{deg} \left( \frac{\nabla_{(t,p)} \Big( (1-s)H_k + s H \Big)}{\left|  \nabla_{(t,p)}  \Big( (1-s)H_k + s H \Big) \right|_\xi}, [t_0 - \gamma, t_0 + \gamma] \times \overline{B_0(\gamma)} , 0 \right) \]
is well-defined and is non-zero by homotopy invariance. Therefore, for $k$ large enough, the mapping $H_k$ in \eqref{conclu8} possesses a critical point $(t_k, p_k) \in [t_0 - \gamma, t_0 + \gamma] \times \overline{B_0(\gamma)}$. By Proposition \ref{generalreduction}, the function defined by:
\ben \label{conclu9}
u_k = u_0 + W_{k,t_k,p_k} + \phi_k(t_k,p_k)
\een
is a critical point of $J$ in \eqref{defJ}. Standard elliptic arguments show then that $u_k$ is in $C^2(M)$, is positive in $M$ and, by the choice of $\eta$ right after Lemma \ref{minor}, that $u_k$ is a positive solution of \eqref{EL}.

To conclude the proof of Theorem \ref{Th1}, we show that the sequence $(u_k)_k$ blows up as $k \to + \infty$ and that for $k$ large enough all the $u_k$ are distinct. First, by the definition of $u_k$ as in \eqref{conclu9} and by the definition of $\phi_k$ as in Proposition \ref{generalreduction} we have, for fixed $R$ large enough:
\ben \label{conclu10}
 \liminf_{k \to + \infty} \Vert u_k \Vert_{L^{\frac{2n}{n-2}}(B_{\xi_k}(R \mk))} > 0,
 \een
where the $\xi_k$ are as in \eqref{defxik}. This shows that the sequence $(u_k)_k$ blows-up in the $C^0$ norm. Then, still by \eqref{conclu9}, by the definition of $W_{k,t,p}$ in \eqref{bulle} and by \eqref{condrk}, for any fixed $R > 0$ and for any integer $l > k$ there holds:
\[ \Vert u_k \Vert_{H^1(B_{\xi_l}(R \mu_l))} = o(1), \]
where $\lim_{k \to \infty} o(1) = 0$. With \eqref{conclu10}, this shows that for $k$ large enough all the $u_k$ are different and concludes the proof of Theorem \ref{Th1}.

\medskip

\noindent Assume then that  $(M,g)$ is not locally conformally flat. Propositions \ref{propenergie} and \ref{propreste} show that there holds:
\be 
\frac{1}{\vek} \left( J \Big( u_0 + W_{k,t,p} + \phi_k(t,p) \Big) - J(u_0) - \frac{1}{n}K_n^{-n} \right) \to H(t,p)
\ee
in $C^1_{loc} \left( (0,+\infty) \times \overline{B_0(1)} \right)$ as $k \to + \infty$, where we have let:
\begin{itemize}
\item If $n=10$:
\be 
\bal
H(t,p) & = - \frac{1}{2^*} \int_{\RR^n} \Psi(p+ty) \left( 1 + \frac{f(\xi_0)}{n(n-2)} |y|^2 \right)^{-n} dy \\
& - \left( \alpha_n + K_n^{-n}\frac{n(n-2)^2}{24(n-4)(n-6)} |W(\xi_0)|_g^2 \right) t^{4},
\eal
\ee
\item if $n \ge 11$:
\be 
\bal
H(t,p) & = - \frac{1}{2^*} \int_{\RR^n} \Psi(p+ty) \left( 1 + \frac{f(\xi_0)}{n(n-2)} |y|^2 \right)^{-n} dy \\
& - K_n^{-n}\frac{n(n-2)^2}{24(n-4)(n-6)} |W(\xi_0)|_g^2 t^{4}.
\eal
\ee
\end{itemize}
Here also we have let $\alpha_n = (n-2)^{\frac{n}{2}} n^{\frac{n-2}{2}} $. Then, a straightforward adaptation of the arguments developed when $7 \le n \le 9$ or $(M,g)$ is locally conformally flat allows to conclude the proof of Theorem \ref{Th1} in the same way in this case.

\section{Proof of Theorem \ref{Th2}} \label{negale6}

\noindent In this section we prove Theorem \ref{Th2}. Most of the computations in the $6$-dimensional case follow from straightforward adaptations of the arguments we developed for the $n \ge 7$ case, so we will only sketch the steps of the proof. We let $(M,g)$ be a closed $6$-dimensional Riemannian manifold of positive Yamabe type. Let $a_0 > 0$ be a smooth function in $M$. By standard sub- and super-solution arguments, there exists a unique solution $u_0$ of the following equation in $M$:
\be
\triangle_g u_0 + \frac{1}{5} S_g u_0 = - u_0^2 + \frac{a_0}{u_0^{4}}.
\ee 
As is easily checked, such a $u_0$ is a strictly stable solution of:
\ben \label{n61}
\triangle_g u_0 + h_0 u_0 = u_0^2 + \frac{a_0}{u_0^4},
\een
where we have let 
\ben \label{n62}
h_0 = \frac{1}{5} S_g + 2 u_0.
\een
We let $H$ be some compactly supported function in $\RR^n$ with $H(0) = 1$ and such that $0$ is a strict local maximum point of $H$. Let $(\vek)_k$ and $(\mk)_k$ be sequences of positive numbers converging to $0$. Define, for $x \in M$:
\ben \label{n63}
\bal
h(x) &= h_0(x) - \sum_{k \ge k_0} \vek H \left( \frac{1}{\mk} \left( \exk \right)^{-1}(x)  \right) \\
a(x) &= a_0(x) -  \sum_{k \ge k_0} \vek H \left( \frac{1}{\mk} \left( \exk \right)^{-1}(x)  \right) u_0^{5}(x),
\eal
\een
where 
\[ \xi_k = \exp_{\xi_0}^{g_{\xi_0}} \left( \left( \frac{1}{k}, 0, \cdots 0 \right)\right)\]
and $\xi_0 \in M$ is arbitrary. Then, for $k_0$ large enough, $\triangle_g +h$ is coercive, $a$ is positive and $u_0$ is again a smooth, positive, strictly stable solution of 
\ben \label{ELn6} 
\triangle_g u + hu = u^2 + \frac{a}{u^4}. 
\een
Let $(\delta_k)_k$ and $(r_k)_k$ be two sequences of positive numbers converging to $0$. Let, for $t > 0$ and $p \in \overline{B_0(1)}$:
\ben \label{defdkyk6}
\left \{ \bal
\delta_k(t) & = \delta_k t \\
y_k(p) & = \exk \left( \mk p \right) .\\
\eal \right.
\een
Define: 
\ben \label{bulle6}
W_{k,t,p}(x) = \Lambda_{y_k}(x) \chi \left( \frac{d_{g_{y_k}}(y_k,x)}{r_k}\right) \delta_k^{2} \left( \delta_k^2 + \frac{f(y_k)}{24}  d_{g_{y_k}}(y_k,x)^2 \right)^{-2},
\een
where $\Lambda$ is as in \eqref{propLambda} and $\delta_k$ and $y_k$ are given by \eqref{defdkyk6}. Assume in addition that:
\ben \label{condrk6}
\delta_k = o \left( \mk \right), \quad \mk = o \left( r_k^2\right) \quad \textrm{ and } \quad r_k = o \left( k^{-2}\right)
\een
as $k \to + \infty$. Let 
\ben \label{defJ6}
J(u) = \frac{1}{2} \int_M \left( |\nabla u|_g^2 + h u^2 \right) dv_g - \frac{1}{3} \int_M (u^+)^{3} dv_g + \frac{1}{3} \int_M \frac{a}{\eta(u)^{3}} dv_g
\een
be the energy functional associated to \eqref{ELn6}. Obviously, the analogous statement of Lemma \ref{minor} remains true in dimension $6$ and $\eta$ in \eqref{defJ6} is chosen so that $u_0$ in \eqref{ELn6} solves also:
\[ \triangle_g u_0 +  h u_0 = u_0^2 + \frac{a}{\eta(u_0)^4}. \]
The general finite-dimensional reduction theorem of Robert-V\'etois \cite{RobertVetois} applies again in this setting and, as in Proposition \ref{generalreduction}, yields the existence of a $C^1$ mapping $\phi_k: (0,M_k) \times \overline{B_0(1)} \to K_{k,t,p}^{\perp}$ such that $u_{k,t,p} = u_0 + W_{k,t,p} + \phi_k(t,p)$ solves  \eqref{ELn6} is and only if $(t,p)$ is a critical point of $(s,q) \mapsto J(u_{k,s,q})$. Since \eqref{propphik} holds again, straightforward computations show that there holds, in $C^0_{loc} \left( (0,+\infty) \times \overline{B_0(1)} \right)$:
\ben \label{n64}
\bal
J\Big(u_0 + W_{k,t,p} + \phi_k(t,p)\Big) -J(u_0) - \frac{1}{6} K_6^{-6} = - 5 H(p) \vek \delta_k(t)^2 \\
+ C_0 a_0(\xi_0) \delta_k(t)^3 + o \left( \delta_k(t)^3 \right) + o \left( \vek \delta_k(t)^2 \right) 
\eal
\een
as $k \to + \infty$, where $J$ is as in \eqref{defJ6} and $C_0$ is some positive constant depending on $u_0(\xi_0)$. Expansion \eqref{n64} follows from \eqref{condrk6} and from the same arguments developed to get \eqref{propener9}, \eqref{propener13} and \eqref{propener15}, \eqref{propener16} and \eqref{propener17}. Choose now $\delta_k = \vek$. Then with \eqref{n64} there holds:
\ben \label{n65}
\bal
\frac{1}{\vek} \left( J\Big(u_0 + W_{k,t,p} + \phi_k(t,p)\Big) -J(u_0) - \frac{1}{6} K_6^{-6} \right)  \\
\to - 5 H(p) t^2 + C_0 a_0(\xi_0) t^3
\eal
\een
in $C^0_{loc} \left( (0,+\infty) \times \overline{B_0(1)} \right)$, as $k \to + \infty$. By the choice of $H$ it is easily seen that the right-hand side in \eqref{n65}, as a function of $(t,p)$, has a strict local minimum at $(t_0, 0)$ where $t_0 = \frac{10}{3 a_0(\xi_0)C_0}$. For large enough $k$, the left-hand side in \eqref{n65} has then a critical point $(t_k, p_k) \in [t_0 - \gamma, t_0 + \gamma] \times \overline{B_0(\gamma)}$ for some $\gamma > 0$. Mimicking the arguments at the end of Section \ref{conclusion} and using Proposition \ref{generalreduction} this concludes the proof of Theorem \ref{Th2}. 

\bibliographystyle{amsplain}
\bibliography{biblio}

\end{document}